\documentclass{amsart}

\usepackage{graphicx}
\usepackage{hyperref}

\newtheorem{theorem}{Theorem}[section]
\newtheorem{proposition}[theorem]{Proposition}

\newtheorem{corollary}[theorem]{Corollary}
\newtheorem{lem}[theorem]{Lemma}

\theoremstyle{definition}
\newtheorem{defn}[theorem]{Definition}

\theoremstyle{remark}
\newtheorem{remark}[theorem]{Remark}

\numberwithin{equation}{section}

\newcommand\meas{\operatorname{meas}}
\newcommand\dist{\operatorname{dist}}
\newcommand{\T}{\mathbb{T}}
\newcommand{\R}{\mathbb{R}}
\newcommand{\Z}{\mathbb{Z}}

\usepackage[OT2,OT1]{fontenc}
\newcommand\cyr
{
	\renewcommand\rmdefault{wncyr}
	\renewcommand\sfdefault{wncyss}
	\renewcommand\encodingdefault{OT2}
	\normalfont
	\selectfont
}
\DeclareTextFontCommand{\textcyr}{\cyr}

\begin{document}

\title{Semiclassical scarring on tori in KAM Hamiltonian systems}

\author{Se\'{a}n Gomes, Andrew Hassell}
\address{}
\curraddr{}
\email{sean.p.gomes@gmail.com}
\thanks{This research was partially supported by the Australian Research Council	through Discovery Projects DP150102419 and DP180100589.}



\date{}



\begin{abstract}
We show that for almost all perturbations in a one-parameter family of KAM Hamiltonians on a smooth compact surface, for almost all KAM Lagrangian tori $\Lambda_\omega$, we can find a semiclassical measure with positive mass on $\Lambda_\omega$. 
\end{abstract}

\maketitle

\section{Introduction}

The purpose of this article is to investigate the phenomenon of scarring (concentration) of sequences of eigenfunctions of quantum systems whose underlying classical system is KAM. 


\subsection{Completely integrable and KAM systems}

Suppose $M$ is a compact boundaryless smooth manifold of dimension $n$.  We work on the cotangent bundle $T^*M$, and consider a Hamiltonian $P(x, \xi)$, that is, a real $\mathcal{C}^\infty$ function on $T^*M$ tending to $+\infty$ as the fibre variable $\xi$ tends to infinity (so that the level sets of $P$ are compact).
The natural symplectic form on $T^* M$ induces Hamiltonian flow with respect to $P$. This dynamical system is said to be completely integrable if there is a symplectic transformation to `action-angle' variables $(I, \theta)$, where $I \in D^n$ lies in some closed ball in $\R^n$ and $\theta$ takes values in $\T^n$, such that the induced Hamiltonian in these coordinates is a function only of $I$, say $H^0(I)$. This transformation can be local in the $I$ variable but must be global in the $\theta$ variable. Then Hamilton's equations of motion in the action-angle variables take the simple form
$$
\dot I = 0, \quad \dot \theta = \omega(I) := \frac{\partial H^0(I)}{\partial I}.
$$
That is, the orbits are restricted to Lagrangian tori $\{ I = \text{constant} \}$, and the motion is quasiperiodic on each torus, with frequency $\omega(I)$. 
%
%
Under the non-degeneracy assumption that the Hessian $\nabla_I^2 H^0$ is non-singular, the tori can be indexed (locally) by frequency $\omega\in\Omega$ rather than action $I$, and we use the notation $\Lambda_\omega$ for this purpose.

If we now consider a smooth one-parameter family of perturbations
\begin{equation}
\label{symbolaa}
H(\theta,I;t)\in\mathcal{C}^\infty(\T^n\times D \times (-1,1)), \quad H(\theta,I;0)=H^0(I),
\end{equation}
it is natural to ask  whether there are any such invariant Lagrangian tori that survive the perturbation for sufficiently small $t$. This problem was resolved by the work of Kolmogorov, Arnold and Moser \cite{kolmogorov},\cite{arnoldkam},\cite{moser}, with the development of what has come to be known as KAM theory.

The initial significant breakthrough in this problem was due to Kolmogorov \cite{kolmogorov}, with the conclusion that although a dense set of tori is indeed generally destroyed by such a perturbation, a large measure collection of the invariant tori $\Lambda_\omega$ survive, precisely those whose frequencies $\omega\in \Omega$ of quasiperiodic flow satisfy the Diophantine condition
\begin{equation}
\label{nonres}
\omega\in \Omega_\kappa=\{\omega\in \Omega:|\langle \omega,k\rangle| \geq \frac{\kappa}{|k|^\tau}\textrm{ for all } k\in \mathbb{Z}^n\setminus \{0\}\textrm{ and } \textrm{dist}(\omega,\partial \Omega)\geq \kappa\}
\end{equation}
where $\kappa >0$ is fixed and $\tau>n-1$.
The tori with frequencies satisfying this Diophantine condition are said to be \emph{nonresonant}.

In the early 2000s, Popov \cite{popovkam} proved a version of the KAM theorem for perturbed completely integrable Hamiltonians in the Gevrey regularity classes $G^\rho(T^*M)$, defined as the set of $u\in \mathcal{C}^\infty(\T^n\times D)$ with

\begin{equation}
\sup_\alpha\sup_{(x,\xi)} L^{-\alpha}\alpha!^{-\rho}|\partial_{x,\xi}^\alpha u|<\infty
\end{equation}
for some $L>0$. For such Hamiltonians, Popov established a Birkhoff normal form analogous to action-angle variables for integrable Hamiltonians.

\vskip 8pt

\subsection{Quantization of KAM systems}

We now turn to the quantum setting. A quantization of the classical system just described is a semiclassical family of pseudodifferential operators $\mathcal{P}_h(t)$, depending on a small semiclassical parameter $h \in (0, h_0]$, and smoothly on a time parameter $t \in [0, t_0]$, with (semiclassical) principal symbol $P(x, \xi; t)$. We shall assume that $\mathcal{P}_h(t)$ has fixed positive differential order and is elliptic and self-adjoint as an operator on half-densities in $L^2(M;\Omega^{1/2})$. Under these conditions, $L^2(M;\Omega^{1/2})$ equipped with the canonical inner product has an orthonormal basis of eigenfunctions of $\mathcal{P}_h(t)$ for each $h \in (0, h_0]$ and $t \in [0, t_0]$. We are interested in the behaviour of these eigenfunctions in the semiclassical limit $h \to 0$, in which we can expect to see properties of the classical dynamical system become visible.  

Using his Birkhoff normal form construction, Popov constructed a so-called quantum Birkhoff normal form for a class of semiclassical differential operators $\mathcal{P}_h(t)$ with principal symbol $P(x,\xi;t)$ and vanishing subprincipal symbol, for  sufficiently small $t$\cite{popovquasis}. The key ingredient we require is the construction of a family of quasimodes for the semiclassical pseudodifferential operator $\mathcal{P}_h$ with exponentially small error term localising onto the nonresonant tori in \cite{popovquasis}. We shall make extensive use of his construction in the present paper (although we only require an error term of the form $O(h^4)$ for our results to go through).  

We formulate our results for semiclassical pseudodifferential operators. Thus $\mathcal{P}_h(t)$ is assumed to be a family of elliptic, self-adjoint semiclassical pseudodifferential operators of fixed positive order $m > 0$. In addition we assume that $\mathcal{P}_h(t)$ has semiclassical principal symbol independent of $h$ and vanishing subprincipal symbol, in the semiclassical sense (that is, the full semiclassical symbol agrees with the principal symbol up to $O(h^2)$).  One example to keep in mind is that of linear self-adjoint perturbations of completely integrable Schr\"{o}dinger Hamiltonians, in which case our operator $\mathcal{P}_h(t)$ has symbol
\begin{equation}
\label{symbol}
\sigma(\mathcal{P}_h(t))=P(x,\xi;t)= \sum_{i,j} g^{ij}(x) \xi_i \xi_j+V(x)+tQ(x,\xi)
\end{equation}
with $V,Q\in G^\rho(T^*M)$, $V$ real valued, $Q$ self adjoint with vanishing subprincipal symbol. For other examples, see Section 5. 

\subsection{Main result.}

Our result is formulated in terms of semiclassical measures. For the reader's convenience we recall the definition here. Suppose that, for a sequence $h_j \downarrow 0$, we have a sequence of functions $u(h_j)$ in $L^2(M;\Omega^{1/2})$, with compact microsupport in the sense that there is a semiclassicial pseudodifferential operator $B$ of semiclassical order $0$ and compact microsupport such that $u(h_j) = B(h_j) u(h_j) + O_{\mathcal{C}^\infty(M;\Omega^{1/2})}(h^\infty)$. 
Let $\nu$ be a positive measure on $T^*M$. We say that $\nu$ is a semiclassical measure associated with the sequence $u(h_j)$ if we have 
\begin{equation}
\lim_{j \to \infty} \langle A_{h_j} u(h_j), u(h_j) \rangle = \int_{T^* M} \sigma(A) d\nu,
\end{equation}
for every semiclassical pseudodifferential operator $A$ of semiclassical order $0$ and compact microsupport. If the $u(h_j)$ are normalized in $L^2(M;\Omega^{1/2})$ then $\nu$ is automatically a probability measure. Compactness theorems show that every normalized sequence $u(h_j)$ with compact microsupport has a subsequence admitting a semiclassical measure. In particular, fixing $t$, this is true for a sequence of normalized eigenfunctions of $\mathcal{P}_{h_j}(t)$ with uniformly bounded eigenvalues as $h_j \to 0$. In the case that the $u(h_j)$ are normalized eigenfunctions, or more generally quasimodes satisfying $(\mathcal{P}_{h_j}(t) - E_j) u(h_j) = o_{L^2}(1)$, $E_j - E \to 0$, then $\nu$ is supported in $T^* M$ on the set $\Sigma_E$ where the symbol $P(x, \xi; t)$ of $\mathcal{P}_h(t)$ is equal to $E$. Suppose for simplicity that $dP(\cdot, t)$ does not vanish on $\Sigma_E$; this implies that $\Sigma_E$ is a smooth codimension $1$ submanifold of $T^*M$. The Liouville measure $\sigma$ on $T^* M$ (viewed as a top-degree form) induces a smooth measure $\lambda_E$ on $\Sigma_E$ by writing $\sigma = \lambda \wedge dP(\cdot, t)$ and then restricting $\lambda$ to $\Sigma_E$. 
Where $\nu$ has positive mass on a set $S \subset \Sigma_E$ of $\lambda_E$-measure zero, we say that the sequence of eigenfunctions \emph{scars}, or \emph{concentrates}, at $S$. 

Popov's quasimode construction yields quasimodes associated to semiclassical measures $\nu$ supported on a single Lagrangian torus $\Lambda_\omega$ for any nonresonant $\omega$. This leads to the question (which was posed to us by S. Zelditch about a decade ago) of whether the true eigenfunctions behave similarly. In the present article, we show that in dimension $n=2$, for almost all $t \in [0, t_0]$ and for a full measure set of invariant tori $\Lambda_\omega$, there are semiclassical measures for $\mathcal{P}_h(t)$ with positive mass on $\Lambda_\omega$. Since the energy surfaces $\Sigma_E$ have dimension $3$ and the Lagrangian tori have dimension $2$ in this case, this shows the existence of sequences of eigenfunctions that scar at $\Lambda_\omega$. More precisely, we prove the following result. 



\begin{theorem}[Main Theorem]
	\label{mainthm}
	Suppose $M$ is a compact boundaryless $G^\rho$ surface, and suppose that $\mathcal{P}_h(t)$ is a family of self-adjoint elliptic semiclassical pseudodifferential operators acting on $\mathcal{C}^\infty (M;\Omega^{1/2})$ with fixed positive differential order $m$, such that 
	\begin{itemize}
		\item The operator $\mathcal{P}_h(t)$ has full symbol real-valued and in the Gevrey class $S_\ell(T^*M)$ from Definition \ref{selldef} where $\ell=(\rho,\mu,\nu)$, with $\rho(\tau+n)+1>\mu>\rho'=\rho(\tau+1)+1$ and $\nu=\rho(\tau+n+1)$;
		\item The principal symbol of $\mathcal{P}_h(t)$ is given by some  $P(x,\xi;t)\in G^{\rho,1}(T^*M\times (-1,1))$;
		\item The Hamiltonian $P^0(x,\xi):=P(x,\xi;0)$ is, in some open set of phase space $T^* M$, non-degenerate and completely integrable; 
		\item Written in action-angle coordinates $(\theta, I) \in \T^n \times D$ for the Hamiltonian $P^0$, the vector fields 
		\begin{multline}
		\nabla_I H^0(I) \text{ and } \nabla_I \left(\int_{\T^2}\partial_t H(\theta,I;0)\, d\theta\right) \\ \text{ are linearly independent for $I\in D$ and all $h<h_0$,}
		\label{li-cond}\end{multline}
		where $H(\theta,I;t)$ denotes $P(x,\xi;t)$ written in the action-angle coordinates for $P^0$, and $H^0(I):=H(\theta,I;0)$.
	\end{itemize}
	
	Then there exists $t_0>0$ such that for almost all $t \in [0,t_0]$, and for almost all KAM tori $\Lambda_\omega=\T^n\times \{I_\omega\}$ with $\omega\in\Omega_\kappa$, there exists a semiclassical measure associated to the eigenfunctions of $\mathcal{P}_h(t)$ 
	that has positive mass (and hence scars) on $\Lambda_\omega$.
\end{theorem}

\begin{remark}
	In \cite{nqe}, under similar assumptions in dimension $n$, the weaker result is shown that $\mathcal{P}_h(t)$ is not quantum ergodic for a full measure set of parameter values $t$.
\end{remark}

\begin{remark}
	As in \cite{nqe}, the key technique is the exploitation of the variation of eigenvalues in the parameter $t$, together with a construction of quasimodes that concentrate entirely on particular KAM tori.
	
	The improvement made by this theorem comes from the fact that $H^0(I)$ and $ (2\pi)^{-2}\int \partial_t H(\theta,I;0) d\theta$ are the leading order terms for the quasieigenvalues and their $t$-derivatives at $t=0$ respectively. In dimension $2$, under the assumptions in Theorem \ref{mainthm}, the level curves of these two quantities intersect transversally and form a coordinate system for the action space $D$. Thus, postponing precise definitions until Section \ref{qbnfsec}, if two quasieigenvalues $\mu_m,\mu_n$ and their time derivatives are both close at some small $t$, then so are the associated actions $I_m$ and $I_n$. This allows us to control spectral clustering of eigenvalues for most values of $t$, which is the key difficulty in passing from properties of quasimodes to properties of true eigenfunctions. 
\end{remark}

\subsection{Outline of this paper}

In Section 2 we review the statement of the quantum Birkhoff normal form, and the resulting explicit expression for quasimodes and quasieigenvalues. This is essentially contained in Popov \cite{popovquasis}, adapted to allow a 1-parameter family parametrized by `time' $t$. 

In Section 3, we prove Theorem \ref{mainthm}. In Section \ref{mainsec2} we use the nonresonance condition \eqref{nonres} to show  that distinct quasieigenvalues typically (that is, for most values of $t$) have spacing bounded below by $h^\gamma$, for some fixed parameter $\gamma \geq 4$,  excluding a family of negligible proportion as $h\rightarrow 0$. 
Using this in Section \ref{mainsec3}, we are able to construct a large family of energy windows $[\mu-h^\gamma,\mu+ch^\gamma]$ about quasieigenvalues in which we control the spectral concentration, in the sense that there are a bounded number of actual eigenvalues in each such window. Applying elementary spectral theory shows that the maximal value of $|\langle u,v \rangle|$ is bounded below by a positive constant independent of $h$, where $v$ is the quasimode with quasieigenvalue $\mu$ and $u$ ranges over the eigenfunctions associated to this energy window. 
Because we have this for all such energy windows except for a family of negligible proportion, as $h \to 0$, we can extract a subfamily, indexed by a sequence $h_j$ tending to zero, associated to quasimodes that concentrate on almost every invariant KAM torus $\Lambda_\omega$. Choosing eigenfunctions $u(h_j)$ so that $|\langle u(h_j),v(h_j) \rangle|$ is bounded below by a positive constant,  we then obtain a sequence of eigenfunctions $u(h_j)$ with positive semiclassical mass on $\Lambda_\omega$.

In Section 4, we remark on our theorem in the setting of $\mathcal{C}^\infty$, as opposed to Gevrey, manifolds. Our choice of Gevrey regularity was pragmatic, based on the availability of the full details of the KAM argument in Popov's papers. We remark that Gevrey regularity, as opposed to analyticity, is flexible enough to allow the use of cutoff functions, which is convenient in designing examples to which our results apply. We give several such examples in Section 5. 

The paper concludes with an appendix, containing definitions of the Gevrey classes and the corresponding pseudodifferential calculus.

\subsection{Related literature} This article is a direct continuation of the research begun by Popov on quasimodes for KAM systems, which has already been discussed. Previously, quasimodes associated to Lagrangian tori were introduced by Colin de Verdi\`ere \cite{colin}. A key component of the argument is a quantum Birkhoff normal form. Extensive use has been made of quantum Birkhoff normal forms when estimating eigenvalues and/or eigenfunctions. We do not attempt a complete review of this literature here, but we mention results for eigenvalues of Schrodinger operators near a minimum value of the potential \cite{sjostrand}; nonself-adjoint operators in two dimensions, in which nonresonant tori also play a key role \cite{melin,HSN}; magnetic Laplacians \cite{ngoc}; and subLaplacians \cite{CHT}. They have also been used in inverse spectral problems, related to wave trace invariants \cite{guillemin, zelditchwave, iant}. 

The idea of using the spectral flow of a 1-parameter family of operators to control spectral concentration for most values of the parameter originates from a paper by the second author \cite{hassellque}, and has been used also by the first author in \cite{mushroom},\cite{nqe}.


.





\section{Quantum Birkhoff Normal Form}
\label{qbnfsec}

We first recall the quantum Birkhoff normal form for the quantization of Gevrey KAM Hamiltonians, originally due to Popov in \cite{popovquasis}. This construction yields exponentially accurate quasimodes localising onto the invariant KAM tori.

We let $M$ be a compact $G^\rho$-smooth manifold of dimension $n\geq 2$ and let $P(x,\xi)=P^0(x,\xi)+P^1(x,\xi)$ be a small $G^\rho$ perturbation of a completely integrable $G^\rho$ Hamiltonian.
From the Liouville-Arnold theorem \cite{arnoldmechanics}, we can write $P$ as
\begin{equation}
(P\circ\chi_1)(\theta,I)=H^0(I)+H^1(\theta,I)
\end{equation}
in the system of action-angle coordinates for the completely integrable Hamiltonian $P^0$.

From the construction in \cite{popovkam}, the Hamiltonian $H(\theta,I)=P\circ \chi_1$  can be placed in a $G^{\rho,\rho(\tau+1)+1}$ Birkhoff normal form about a family of invariant tori $\{\Lambda_\omega\}$ with frequencies $\omega\in \Omega_\kappa$. The precise definition of the anisotropic Gevrey classes $G^{\rho,\rho'}$ can be found in Definition \ref{anisgevdef}.
The existence of a Birkhoff normal form means that we can write
\begin{equation}
\label{bnf}
\tilde{H}(\theta,I)=H\circ \chi = K(I)+R(\theta,I)
\end{equation}
where $R$ is flat at the set of nonresonant actions $E_\kappa=\omega^{-1}(\Omega_\kappa)$ for a suitable choice of $G^{\rho,\rho(\tau+1)+1}$ exact symplectic transformation $\chi:\T^n\times D \rightarrow \T^n\times D$ with $D\subset \R^n$ compact. 
In particular, one can apply this result to the one-parameter family of Hamiltonians \eqref{symbolaa}.
In this case, we obtain a family of exact symplectic transformations $\chi_t:\T^n\times D \rightarrow \T^n\times D$ that transform the Hamiltonian $H(\theta,I;t)$ into the Birkhoff normal form
\begin{equation}
\tilde{H}(\theta,I;t)=H\circ \chi_t = K(I;t)+R(\theta,I;t)
\end{equation}
Furthermore, from \cite[Proposition 3.14]{nqe}, the components $K(I;t),R(\theta,I;t)$ of the normal form have smooth dependence on the parameter $t$ and we have
\begin{equation}
\label{deriv1}
K(I;t)=H^0(I)+t\partial_t H(\theta,I;0)+O(|t|^{9/8})
\end{equation}
uniformly in $I\in D$.



Now fix $\ell=(\rho,\mu,\nu)$ with $\rho>1$, $\rho(\tau+n+1)>\mu>\rho(\tau+1)+1$ and $\nu=\rho(\tau+n+1)$ and let $\mathcal{P}_h(t)$ be an smooth $1$-parameter family of formally self-adjoint semiclassical pseudodifferential operators with full symbols in the Gevrey class $S_\ell$, acting on half-densities with principal symbol $P$ and vanishing subprincipal symbol. One then obtains a quantum Birkhoff normal form in the class of Gevrey semiclassical pseudodifferential operators.

\begin{theorem}
\label{main2}
There exists a family of semiclassical Fourier integral operators 
\begin{equation}
U_h(t):L^2(\T^n;\mathbb{L})\times (-1,1)\rightarrow L^2(M)\quad (0<h<h_0)
\end{equation}
that are uniformly bounded in $t,h$ and are associated with the canonical relation graph of the Birkhoff normal form transformation $\chi_t$ such that for each fixed $t\in (-1,1)$, we have
\begin{enumerate}
\item $U_h(t)^*U_h(t)-\textrm{Id}$ is a pseudodifferential operator with symbol in the Gevrey class $S_\ell(\T^n\times D)$ which restricts to an element of $S_\ell^{-\infty}(\T^n\times Y)$ for some subdomain $Y$ of $D$ that contains $E_\kappa(t).$
\item $\mathcal{P}_h(t)\circ U_h(t)-U_h(t)\circ \mathcal{P}^0_h(t)=\mathcal{R}_h(t)\in S_\ell^{-\infty}$, where the operator $\mathcal{P}^0_h(t)$ has symbol 
\begin{equation}
\label{fioconjflatness}
p^0(\theta,I;t,h)=K^0(I;t,h)+R^0(\theta,I;t,h)=\sum_{j\leq \eta h^{-1/\nu}}K_j(I;t)h^j+\sum_{j\leq \eta h^{-1/\nu}}R_j(\theta,I;t)h^j
\end{equation}
with both $K^0$ and $R^0$ in the symbol class $S_\ell(\T^n\times D)$ from Definition \ref{selldef} where $\eta>0$ is a constant, $K_0(I;t),R_0(\theta,I;t)$ are the components of the Birkhoff normal form of the Hamiltonian $P_0\circ \chi_1$, and
\begin{equation}
\label{qbnfflatness}
\partial_I^\alpha R_j(\theta,I;t)=0 
\end{equation}
for $(\theta,I;t)\in \T^n\times E_\kappa(t)\times (-1,1)$.
\end{enumerate}
\end{theorem}
Here $\mathbb{L}$ denotes the Maslov line bundle associated to the embedded KAM Lagrangian tori. See \cite{popov2}, \cite{popovquasis}.

This quantum Birkhoff normal form was obtained in \cite{popovquasis} without the presence of the parameter $t$. In \cite{nqe} Section 4, the same construction is carried out with the presence of the parameter $t$. In particular, it is noted that the symbols $K_j,R_j$ can be taken smooth in $t$.

As a consequence of Theorem \ref{main2}, one obtains a $t$-dependent family of Gevrey class quasimodes as is shown in \cite{popovquasis} Section 2.4. 
In particular, for each $t$ we obtain a finite $h$-dependent family $u_m(t,h)\in \mathcal{C}_c^\infty(M)$ supported in a  bounded $h$-independent domain such that
\begin{equation}
\|\mathcal{P}_h u_m(t,h)-\mu_m(t,h)\|=O(\exp(-ch^{-1/\rho(\tau+n+1)}))
\end{equation}
and
\begin{equation}
|\langle u_m(t,h),u_n(t,h)\rangle-\delta_{mn} |=O(\exp(-ch^{-1/\rho(\tau+n+1)}))
\end{equation}
where the index set is
\begin{equation}
\mathcal{M}_h(t):=\{m\in \mathbb{Z}^n:\textrm{dist}(h(m+\vartheta/4),E_\kappa(t))<Lh\}
\end{equation}
for a fixed $L>0$, where $E_\kappa(t)=\omega^{-1}(E_\kappa;t)$ is the collection of nonresonant actions and $\vartheta/4$ is the Maslov class of the embedded Lagrangian tori in $\T^*M$. (See \cite{popov2} \cite{popovquasis} for details.)

These quasimodes satisfy the asymptotic
\begin{equation}
\mathcal{M}_h(t)\sim (2\pi h)^{-n}\meas(E_\kappa(t)).
\end{equation}
and for each fixed $t$, these quasimodes $\mathcal{Q}$ have all associated semiclassical measures  supported on the nonresonant invariant Lagrangian tori $\Lambda_\omega$ with $\omega\in \Omega_\kappa$.
In terms of the quantum Birkhoff normal form from Theorem \ref{main2}, the quasimodes are given by
\begin{equation}
\label{popovquasisexpression}
(v_m(t,h),\mu_m(t,h))=(U_h(t)e_m,K^0(I_m,t;h))
\end{equation}
where $I_m:=h(m+\vartheta/4)$ for $m\in \mathcal{M}_h(t)\subset\mathbb{Z}^n$, $\{e_m\}_{m\in \mathbb{Z}}$ is the orthonormal basis of $L^2(\T^n;\mathbb{L})$ associated with the quasiperiodic functions
\begin{equation}
\tilde{e}_m(x):=\exp(i(m+\vartheta/4)\cdot x)
\end{equation}
on $\R^n$ and
\begin{equation}
K^0(I,t;h):=\sum_{j \leq Ch^{-1/\nu}}K_j(I,t)h^j
\end{equation}
is the integrable part of the quantum Birkhoff normal form, with $K_0=K$ and $K_1=0$.


\begin{remark}
	Typically, fixed $I_m\in h(\Z^n+\vartheta/4)$ will only be in $E_\kappa(t)$ for $O(h)$-sized intervals as $t$ varies.
\end{remark}

By truncating the symbol expansion of the elliptic symbol $a$ in \cite[Proposition~4.2]{nqe} to some finite order error $O(h^{\gamma+1})$, we have that the quantum Birkhoff normal form symbols $K^0$ and $R^0$ have expansions truncated to the same finite order, at the cost of enlarging the error term $\mathcal{R}_h(t)$ in Theorem \ref{main2} to order $O(h^{\gamma+1})$. This weakens the error estimate in the quasimodes \eqref{popovquasisexpression} to $O(h^{\gamma+1})$. Such quasimodes with $\gamma \geq 4$ are sufficient for the application in this paper. 

The results in Section \ref{mainsec} rely on also being able to find a bound for $K^0$ and its time derivatives that is uniform in $(t,h)\in (0,t_0)\times (0,h_0)$. Since we have truncated the series expansion of $K^0$ to finite order error  $O(h^{\gamma+1})$, this follows easily from smoothness of the principal symbol $K_0$ and the fact the the homological equation used to iteratively solve for $K_j$ (see \cite[Proposition~4.3]{nqe}) preserves smoothness in $t$ and gives us explicit uniform bounds on the time derivatives.


\section{Scarring on individual KAM tori}
\label{mainsec}
We now set about proving Theorem \ref{mainthm}. We begin by fixing a one-parameter family of perturbations
\begin{equation}
\label{pertform}
H(\theta,I;t)\in G^{\rho,\rho,1}(\T^2\times D\times (-1,1))
\end{equation}
of the nondegenerate completely integrable Hamiltonian $H^0(I)=H(\theta,I;0)$. We assume without loss of generality that $D$ is convex by shrinking if necessary, and fix KAM parameters $\tau=2$ (this is an arbitrary but convenient choice, any $\tau > 1$ will do) and $\kappa>0$. We also choose $\kappa$ sufficiently small so that the set of nonresonant frequencies $\Omega_\kappa$ has positive measure. 

We also make the geometric assumption \eqref{li-cond} on the perturbation family $H(\theta,I;t)$.

\begin{remark}
	Notice that if this condition is satisfied at one action $I_*$, then it is satisfied in a neighbourhood, so can be assumed throughout $D$ by shrinking $D$ if necessary. It is clear that the set of perturbations $H$ satisfying the condition at one point is a codimension one set. In this sense (i.e.\ shrinking $D$ as necessary) the geometric assumption holds generically. 
\end{remark}

This assumption implies that the function $K^0$ (the integrable part of the quantum Birkhoff normal form) and its time derivative $\partial_t K^0$ locally form coordinates in $D$ for all $t<t_0$ and $h<h_0$.

More precisely, we have:
\begin{proposition}
	\label{geoprop}
	There exists $h_0,t_0>0$ such that for all $0<t<t_0$ and all $0<h<h_0$, we have that
	\begin{equation}
	\eta:I\mapsto (K^0(I,t;h),\partial_tK^0(I,t;h))
	\end{equation}
	is a local diffeomorphism, with
	\begin{equation}
	G_1|\eta(I_1)-\eta(I_2)|\leq |I_1-I_2|\leq G_2|\eta(I_1)-\eta(I_2)|
	\end{equation}
	for some positive constants $G_1,G_2$ that depend on our choice of perturbation $H$ but are uniform in $t$ and $h$.
\end{proposition}

\begin{proof}
	From \eqref{deriv1} and the finite symbolic expansion
	\begin{equation}
	K^0(I,t;h)=\sum_{j\leq \gamma} K_j(I,t)h^j
	\end{equation}
	with each $K_j$ smooth, it follows that
	\begin{equation}
	\partial_tK^0(I,0;h)=(2\pi)^{-2}\int_{\T^2} \partial_tH(\theta,I;0)\, d\theta+O(h)
	\end{equation}
	uniformly in $I$.
	Hence 
	\begin{equation}
	\partial_t K^0(I,t;h)=(2\pi)^{-2}\int_{\T^2} \partial_tH(\theta,I;0)\, d\theta+O(t)+O(h)
	\end{equation}
	Moreover,
	\begin{equation}
	K^0(I,t;h)=K_0(I,t)+O(h)=K_0(I,0)+O(t)+O(h).
	\end{equation} 
	Hence the claim follows from the linear independence of 
	\[\nabla H^0(I)\] and \[\nabla \left(\int_{\T^2}\partial_tH(\theta,I;0)\, d\theta\right)\]
	by taking $t_0,h_0$ sufficiently small.
\end{proof}

\subsection{Non-concentration of quasi-eigenvalues}
\label{mainsec2}
In Section \ref{qbnfsec}, we introduced a family of quasimodes for the quantization $H(x,hD)$ using the quantum Birkhoff normal form \ref{main2}.
In particular, the quasieigenvalues were given in terms of the quantum Birkhoff normal form by
\[\mu_m(t;h)=K^0(I_m,t;h) \]
where
$I_m=h(m+\vartheta/4)$ for $m\in \mathcal{M}_h(t)$. We write $\mu_m(t)=K^0(I_m,t;h)$ even when $I_m \notin \mathcal{M}_h(t)$.

A consequence of the nonresonance condition \eqref{nonres} is a lower bound on the difference between quasieigenvalues associated to actions $I_m,I_n$ with a small difference.
\begin{proposition}
	\label{nonresonancenonconc}
	There exist constants $C_1,C_2>0$ dependent on our choice of perturbation $H$ and on the nonresonance constant $\kappa$ but independent of $t$ and $h$, such that for all $m,n\in \Z^2$ with $I_m,I_n\in D$ such that
	\begin{equation}
	|I_m-I_n|\leq C_1h^{3/4}
	\end{equation}
	and $n\in\mathcal{M}_h(t)$ we have
	\begin{equation}
	\label{mudiff}
	|\mu_m-\mu_n|\geq C_2h^{3/2}.
	\end{equation}
\end{proposition}
\begin{proof}
	First, by taking the leading order term in the semiclassical expansions of the $K^0$, we have
	\begin{equation}
	\label{leading}
	|\mu_m-\mu_n|\geq |K_0(I_m,t)-K_0(I_n,t)|+O(h^2)
	\end{equation}
	uniformly for $t<t_0$.
	Taylor expansion yields
	\begin{equation}
	\label{nonresest}
	K_0(I_m,t)-K_0(I_n,t)=h\nabla K_0 (I_n)\cdot(m-n)+h^2 \langle \nabla^2 K_0(\tilde{I},t)(m-n),(m-n) \rangle 
	\end{equation}
	for some $\tilde{I}$ on the line segment between $I_m$ and $I_n$.
	
	Since $I_n\in\mathcal{M}_h(t)$, we also have
	\begin{equation}
	|\nabla K_0(I_n)-\nabla K_0(I_\omega)|=O(h)
	\end{equation}
	uniformly for $t< t_0$ where $I_\omega$ is some nonresonant action corresponding to a nonresonant frequency $\omega\in\Omega_\kappa$.
	Inserting this estimate into \eqref{nonresest}, we obtain
	\begin{eqnarray}\nonumber
	K_0(I_m,t)-K_0(I_n,t)&=&h\nabla K_0(I_\omega)\cdot (m-n)+O(h^2|m-n|^2)+O(h^2|m-n|)\\
	\nonumber&=&h\nabla K_0(I_\omega)\cdot (m-n)+O(h^2|m-n|^2)\\
	\nonumber &\geq & \frac{h\kappa}{|m-n|^2}+O(h^2|m-n|^2)\\
	&\geq & (\kappa C_1^{-2} +O(C_1^2))h^{3/2}\label{balance}
	\end{eqnarray}
	by bounding the leading term below using the nonresonance condition \eqref{nonres}.
	
	The claim now follows from \eqref{leading} and \eqref{balance} upon choosing $C_1$ suitably small.
	
\end{proof}
Qualitatively, Proposition \ref{nonresonancenonconc} shows that if two distinct quasieigenvalues $\mu_m,\mu_n$ are very close (that is, less than $C_2 h^{3/2}$ apart), then there is a lower bound on how close their actions $I_m,I_n$ can be (they must differ by at least $C_1 h^{3/4}$). Applying Proposition \ref{geoprop}, this in fact gives us a lower bound (of the order of $h^{3/4}$)  on the difference of speeds $\partial_t(\mu_m-\mu_n)$. This forces them to separate quite quickly as $t$ evolves. This is quantified in the following proposition.

\begin{proposition}
	\label{separation} Choose any $\gamma > 3/2$. 
	Suppose that $h<h_0$, $m,n\in \Z^2$, and $t_*\in(0,t_0)$ are fixed with $I_m,I_n\in D$, $m\in\mathcal{M}_h(t_*)$ and
	\begin{equation}
	|\mu_m(t_*,h)-\mu_n(t_*,h)|<h^\gamma<C_2h^{3/2}.
	\end{equation}
	If we denote 
	\begin{equation}
	\mathcal{C}_{m,n}(h)=\{t\in (0,t_0):|\mu_m-\mu_n|<h^\gamma \},
	\end{equation}
	then there exist positive constants $\tilde{C}_1,\tilde{C}_2$ which depend on the constants $C_1,C_2$ from Proposition \ref{nonresonancenonconc} as well as the geometric constants   $G_1,G_2$ from Proposition \ref{geoprop} such that
	\begin{equation}
	\frac{\meas\big( [t_*-\tilde{C}_1h^{3/4},t_*+\tilde{C}_1h^{3/4}]\cap \mathcal{C}_{m,n}(h) \big)}{h^{3/4}}< \tilde{C}_2h^{\gamma-3/2}.
	\label{Cmn}	\end{equation}
\end{proposition}

\begin{proof}
	From Proposition \ref{geoprop}, we have that 
	\begin{equation}
	|\partial_t \mu_m(t,h)-\partial_t \mu_n(t,h)|=|\partial_t K^0(I_m,t;h)-\partial_t K^0(I_n,t;h)|\geq Ch^{3/4}.
	\end{equation}
	where $C$ depends on $C_1,C_2,\kappa,L$ and the geometric constants $G_i$.
	
	By Taylor expanding we have
	\begin{equation}
	\mu_m(t;h)=K^0(I_m,t_*;h)+(t-t_*)\partial_t(K^0(I_m,t_*;h))+O(|t-t_*|^2)
	\end{equation}
	with error term uniform in $h$ and $m$. It follows that 
	\begin{equation}
	|\mu_m(t;h) - \mu_n(t;h) | =(t-t_*)|\partial_t K^0(I_m,t;h)-\partial_t K^0(I_n,t;h)| + O(h^\gamma) + O(|t-t_*|^2).
	\end{equation}
	By choosing $\tilde C_1$ sufficiently small, the quadratic term $O(|t-t_*|^2)$ is dominated by the linear term for $|t-t_*| \leq \tilde C_1 h^{3/4}$. Also, the $h^\gamma$ term is dominated by the others for sufficiently small $h$ since $\gamma > 3/2$. It follows that we have 
	\begin{equation}
	|\mu_m(t;h) - \mu_n(t;h) | \geq \frac1{2} (t-t_*)|\partial_t K^0(I_m,t;h)-\partial_t K^0(I_n,t;h)| 
	\end{equation}
	for $t \in [t_*-\tilde{C}_1h^{3/4},t_*+\tilde{C}_1h^{3/4}]$. Hence, we only have $t \in \mathcal{C}_{m,n}(h)$ for $|t-t_*| \leq C^{-1}  h^{\gamma - 3/4}$. This yields \eqref{Cmn}, where the the constants $\tilde{C}_i$  depend on the original Hamiltonian, the perturbation, $\kappa$, $L$, and the $G_i$, but not on $t$ or $h$.
\end{proof}

From Proposition \ref{separation}, we can deduce that, provided that $\gamma$ exceeds $7/2$,  for any fixed index $m\in \Z^2$ with ${I_m\in D}$, and for any $t\in (0,t_0)$ for which $m\in\mathcal{M}_h(t)$, $\mu_m(t,h)$ is typically the only quasieigenvalue in a window of size $O(h^\gamma)$.

To state this consequence precisely, we introduce some new notation. For each $m\in \Z^n$ such that $I_m\in D$, we define
\begin{equation}
A_m:=\{t\in(0,t_0):m\in\mathcal{M}_h(t)\}
\end{equation}
and
\begin{equation}
B_m:=\{t\in (0,t_0):m\in\mathcal{M}_h(t) \textrm{ and }|\mu_m(h)-\mu_n(h)|>h^\gamma \textrm{ for all }n\neq m \textrm{ with }I_n\in D\}.
\end{equation}

\begin{proposition}
	\label{nonconc1}
	Assume $\gamma > 7/2$. For $A_m,B_m$ defined as above, we can choose $\epsilon(h)>0$ such that we have
	\begin{equation}
	\meas(B_m) \geq (1-\epsilon(h)^2)\meas(A_m)
	\end{equation}
	with
	\begin{equation}
	\epsilon(h)=O(h^{\gamma/2-7/4})
	\end{equation}
\end{proposition}

\begin{proof}
	Proposition \ref{separation} implies that the measure of the set 
	\begin{equation}
	\label{flowspeed}
	\Big\{t\in (0,t_0):m\in \mathcal{M}_h(t)\textrm{ and }|\mu_m(t;h)-\mu_n(t;h)|< h^\gamma\Big\}
	\end{equation}
	is bounded by $C h^{\gamma - 3/2}$ 
	for any $m,n\in\mathbb{Z}^2$ such that $I_m, I_n \in D$, and any $h<h_0$.
	
	The total number of $n$ such that $I_n \in D$ is $O(h^{-2})$. 
	By summing over all such $n$, we obtain an upper bound of $O(h^{\gamma-7/2})$ for the amount of time that $\mu_m(t)$ is a quasi-eigenvalue that is within $h^\gamma$ of another quasi-eigenvalue.
	
	The discussion above implies that
	\begin{equation}
	\label{usuallyisolated}
	\meas(B_m) \geq (1-\epsilon(h)^2)\meas(A_m)
	\end{equation}
	with
	\begin{equation}
	\label{epsdef}
	\epsilon(h)= O(h^{\gamma/2-7/4}).
	\end{equation}
\end{proof}

We can also recast Proposition \ref{nonconc1} as a statement of nonconcentration of quasi-eigenvalues for fixed $t$.
We use the notation
\begin{equation}
N_1(t,h)=\#\mathcal{M}_h(t) = \#\{m\in \mathbb{Z}^2:I_m\in D \textrm{ and } t\in A_m\}
\end{equation}
and
\begin{equation}
N_2(t,h)=\# \mathcal{B}_h, \quad \mathcal{B}_h := \{m\in \mathbb{Z}^2:I_m\in D \textrm{ and } t\in B_m\}. 
\label{BN2}\end{equation}

\begin{proposition}
	\label{goodtime}
	Let $N_1, N_2$ be defined as above. Then the set 
	\begin{equation}
	\label{nonconccond}
	\mathcal{G} := \Big\{ t \in (0, t_0) \mid \exists \text{ sequence } h_j \to 0 \text{ such that } \frac{N_2(t,h_j)}{N_1(t,h_j)}>1-\epsilon(h_j) \Big\}
	\end{equation}
	has full measure in $(0, t_0)$. 
\end{proposition}

\begin{proof}
	By Fubini's theorem, we have
	$$
	\int_0^{t_0} N_2(t,h) \, dt = \sum_{m : I_m \in D} \meas(B_m),
	$$
	with a similar relation for the sets $N_1$ and $A_m$. 
	From \eqref{usuallyisolated}, it follows that
	\begin{equation}
	(1-\epsilon(h)^2)\int_0^{t_0} N_1(t,h) \, dt \leq \int_0^{t_0} N_2(t,h)\, dt.
	\label{NN}\end{equation}
	Let $S_h$ be the set 
	\begin{equation}
	S_h=\{t\in (0,t_0):N_2(t,h)\leq (1-\epsilon(h))N_1(t,h)\}.
	\end{equation}
	Then since $\epsilon(h) N_1(t,h) \leq N_1(t,h)-N_2(t,h)$ for $t \in S_h$, we have 
	\begin{equation}\begin{aligned}
	\epsilon(h)\int_{S_h} N_1(t,h) \leq \int_{S_h} \big( N_1(t,h)-N_2(t,h)\big) \, dt &\leq \int_0^{t_0} \big( N_1(t,h)-N_2(t,h) \big) \, dt \\
	&\leq \epsilon(h)^2 \int_0^{t_0} N_1(t,h) \, dt
	\end{aligned}\end{equation}
	where we used \eqref{NN} in the last step. 
	Consequently,
	\begin{equation}
	\int_{S_h} N_1(t,h) \, dt \leq \epsilon(h) \int_0^{t_0} N_1(t,h) \, dt. 
	\end{equation}
	Since
	\begin{equation}
	N_1(t,h)=\# \mathcal{M}_h(t)\sim h^{-2}\meas(E_\kappa(t))
	\end{equation}
	we deduce that 
	\begin{equation}
	\label{integralbound}
	h^2\int_0^{t_0} 1_{S_h}(t)\cdot N_1(t,h)\, dt \leq h^2\epsilon(h)\int_0^{t_0} N_1(t,h)\, dt=O(\epsilon(h))=o(1).
	\end{equation}
	Fatou's lemma then implies that
	\begin{equation}
	\int_0^{t_0}\left(\liminf_{h\rightarrow 0} 1_{S_h}(t)\cdot \meas(E_\kappa(t))\right)\, dt=0.
	\end{equation}
	As $\meas(E_\kappa(t))$ is bounded away from zero, this shows that $\liminf_{h\rightarrow 0} 1_{S_h}(t)$ vanishes almost everywhere, 
	which completes the proof.
\end{proof}

\subsection{Non-concentration implies positive mass}
\label{mainsec3} We now consider a fixed $t \in \mathcal{G}$ and a fixed sequence $h_j\rightarrow 0$ such that the conclusions of Proposition \ref{goodtime} hold. We supress $t$-dependence of various quantities in our notation in this section for brevity.

Introducing the energy windows
\begin{equation}
\mathcal{W}_m(h):=[\mu_m(h)-h^\gamma/3,\mu_m(h)+h^\gamma/3],
\end{equation}
Proposition \ref{goodtime} implies that for a sequence $h=h_j\rightarrow 0$ we can find a large subcollection $\mathcal{B}_{h_j}\subset \mathcal{M}_{h_j}$ of size $N_2(h_j)$ (see \eqref{BN2}) that indexes a disjoint subcollection of energy windows $\mathcal{W}_m(h_j)$.

We now want to study the number of true eigenvalues $E_j(h)$ lying in the window $\mathcal{W}_m(h_j)$. Let
\begin{equation}
\mathcal{N}_m(h)=\#\{E_j(h)\in \mathcal{W}_m(h)\}.
\end{equation}
From Weyl's law, the total number of eigenvalues in the energy band $[a,b]$ is asymptotic to $(2\pi h)^{-2}\meas(p^{-1}([a,b]))$ where $p=\sigma(\mathcal{P}_h)$, whilst the number of quasimodes in our local patch $\T^n\times D$ that are $O(h^\gamma)$-isolated in energy satisfies the asymptotic $N_2(h_j)\sim N_1(h_j)\sim (2\pi h_j)^{-2}\meas(E_\kappa)$. Typically the ratio \begin{equation}
R:=\frac{\meas(p^{-1}([a,b]))}{\meas(E_\kappa)}
\end{equation}
will be much larger than $1$ if our coordinate patch is very small.

Fixing $\lambda >1$, we define
\begin{equation}
\tilde{\mathcal{B}}_{h}(\lambda):=\{m\in\mathcal{B}_{h}:\mathcal{N}_m(h)<\lambda R\}.
\end{equation}
From the disjointness of the  $\mathcal{W}_m(h)$ for $m\in \mathcal{B}_{h}$, together with the definition of the $\tilde{\mathcal{B}_h}$, we then have
\begin{equation}
\limsup_{j\rightarrow\infty }h_j^2\#(\mathcal{B}_{h_j}\setminus \tilde{\mathcal{B}_{h_j}})\cdot \lambda R \leq \frac{\meas(p^{-1}([a,b]))}{4\pi^2}.
\end{equation}
Since $h_j^2\# \mathcal{B}_{h_j}=h_j^2N_2(h_j)\rightarrow (4\pi^2)^{-1} \meas(E_\kappa)$ as $j\rightarrow\infty$, it follows that
\begin{equation}
\label{windowsboundedcount}
\frac{\#\tilde{\mathcal{B}}_{h_j}(\lambda)}{\#\mathcal{B}_{h_j}}=1-\frac{ \#(\mathcal{B}_{h_j}\setminus \tilde{\mathcal{B}_{h_j}})}{\mathcal{B}_{h_j}}> 1-2/\lambda
\end{equation}
for each sufficiently large $j$.
That is, the proportion of $O(h^\gamma)$-sized energy windows associated to actions in $\mathcal{B}_{h}$ containing at most  $\lambda R$ eigenvalues is at least $1-2/\lambda$. 

\begin{lem}\label{lem:spectral}
	For each $m\in \tilde{A}_{h_j}(\lambda)$,  there exists an eigenfunction $u_{k_j}(h_j)$ with eigenvalue $E_j(h_j) \in [\mu_m-h^\gamma,\mu_m+h^\gamma]$ such that 
	\begin{equation}
	\label{largeinnerprod}
	\max_{|E_j-\mu_m|\leq h^\gamma} |\langle u_j,v_m \rangle|\geq \frac{1-o(1)}{\lambda R}.
	\end{equation}
\end{lem}

\begin{proof} Since the quasimodes are of order $O(h^{\gamma+1})$, we have 
	\begin{eqnarray}
	\nonumber h^{\gamma}\|\pi_{[\mu_m-h^\gamma,\mu_m+h^\gamma]}^\perp(v_m)\|&\leq&\|(T-\mu_m)\pi_{[\mu_m-h^\gamma,\mu_m+h^\gamma]}^\perp(v_m)\|\\
	\nonumber&\leq &\|(T-\mu_m)v_m\|\\
	\nonumber&=&O(h^{\gamma+1})\\
	\Rightarrow \|\pi_{[\mu_m-h^\gamma,\mu_m+h^\gamma]}(v_m)\|&=&1-O(h)
	\end{eqnarray}
	where $\pi_I$ is the spectral projector associated to $\mathcal{P}_h$. This spectral projector can be expressed 
	$$
	\sum_{E_{k_j}(h_j) \in [\mu_m-h^\gamma,\mu_m+h^\gamma]} u_{k_j}(h_j) \langle \cdot, u_{k_j}(h_j) \rangle,
	$$
	from which \eqref{largeinnerprod} follows. 
\end{proof}
We now show that most nonresonant actions in $E_\kappa$ of KAM tori are close to actions in $\tilde{\mathcal{B}}_{h_j}(\lambda)$ for all sufficiently large $j$. This shows that the concentrating quasimodes associated to such torus actions be formed from the subfamily ${\tilde{\mathcal{B}}_{h_j}(\lambda)\subset \mathcal{M}_{h_j}}$. We introduce the notation 
\begin{equation}\begin{aligned}
\tilde{\mathcal{I}}_{h_j}(\lambda) &= \{ h(m + \vartheta/4) \mid m \in \tilde{\mathcal{B}}_{h_j}(\lambda) \} \\
{\mathcal{I}}_{h_j}(\lambda) &= \{ h(m + \vartheta/4) \mid m \in {\mathcal{M}}_{h_j}\}
\end{aligned}\label{tildeI}\end{equation}
for the actions corresponding to integer pairs $m \in \tilde{\mathcal{B}}_{h_j}(\lambda)$, respectively $m \in {\mathcal{M}}_{h_j}$. 

\begin{proposition}
	Let $\tilde{\mathcal{I}}_{h_j}(\lambda)$ be defined as in \eqref{tildeI}. Then 
	\label{mostactions}
	\begin{eqnarray}
	\frac{\meas \big(\{I\in E_\kappa(t):\dist(I,\tilde{\mathcal{I}}_{h_j}(\lambda))<Lh_j\}\big)}{\meas(E_\kappa)}
	&\geq& 1-\frac{L^2}{\pi \lambda}.
	\end{eqnarray}
	for all sufficiently large $j$.
\end{proposition}
\begin{proof}
	We have
	\begin{eqnarray}
	\nonumber& &\frac{\meas(\{I\in E_\kappa(t):\textrm{dist}(I,\tilde{\mathcal{I}}_{h_j}(\lambda))<Lh_j\})}{\meas(E_\kappa)}\\
	\nonumber&\geq& 1-\frac{\meas(\{I\in E_\kappa(t):\textrm{dist}(I,\mathcal{I}_{h_j}\setminus\tilde{\mathcal{I}}_{h_j}(\lambda))<Lh_j\})}{\meas(E_\kappa)}\\
	\nonumber&\geq & 1-\frac{1}{\meas(E_\kappa)}\cdot \# (\mathcal{M}_{h_j}\setminus\tilde{\mathcal{B}}_{h_j}(\lambda))\cdot \pi L^2 h_j^2\\
	\nonumber&=& 1-\frac{N_1(h_j)}{\meas(E_\kappa)}\cdot  \left(1-\frac{\#\tilde{\mathcal{B}}_{h_j}(\lambda)}{N_2(h_j)}\cdot\frac{N_2(h_j)}{N_1(h_j)}\right)\cdot \pi L^2 h_j^2\\
	&\geq & 1-\frac{N_1(h_j)}{\meas(E_{\kappa})}\cdot (1-(1-2\lambda^{-1})(1-\epsilon(h_j)))\cdot \pi L^2 h_j^2
	\end{eqnarray}
	from Proposition \ref{goodtime} and \eqref{windowsboundedcount}.
	
	Taking $j\rightarrow \infty$ and using the asymptotic 
	\begin{equation}
	N_1(h)\sim \frac{\meas(E_\kappa)}{(2\pi h)^2}
	\end{equation}
	completes the proof.
\end{proof}

Proposition \ref{mostactions} together with \eqref{largeinnerprod} are the key ingredients required to prove Theorem \ref{mainthm}.

\subsection{Proof of Theorem \ref{mainthm}}

Proposition \ref{mostactions} and the Borel--Cantelli lemma imply  that there exists a subset $\tilde{E}_\kappa(\lambda)\subset E_\kappa$ of the nonresonant actions of proportion at least $1-O(\lambda^{-1})$ that has the property that
\begin{equation}
I\in \tilde{E}_\kappa(\lambda) \Rightarrow \textrm{dist}(I,\tilde{\mathcal{I}}_{h_j}(\lambda))<Lh_j \textrm{ for infinitely many }j.
\end{equation}
For each $I_\omega\in \tilde{E}_\kappa(\lambda)$ and each $j$, we choose such an action in  $\tilde{A}_{h_j}$ and an associated quasimode $v_{m_j}$ for $\mathcal{P}_{h_j}$ in order to obtain a sequence of quasimodes that concentrates completely on the torus $\Lambda_\omega=\{I_\omega\}\times \T^{2} $.

For this sequence, we can find using Lemma~\ref{lem:spectral} a corresponding sequence of eigenfunctions $u_{k_j}$ for $\mathcal{P}_{h_j}$ such that 
\begin{equation}
\label{largeinnerproduct2}
|\langle u_{k_j}(h_j),v_{m_j}(h_j) \rangle|>\frac{1}{2\lambda R}
\end{equation}
for all sufficiently large $j$.

We now claim that the sequence $u_{k_j}(h_j)$ scars on the torus $\Lambda_\omega$. To see this, we take 
an arbitrary semiclassical pseudodifferential operator $A_h$ with compactly supported symbol equal to $1$ in a neighbourhood of the torus $\Lambda_\omega$, and estimate 
\begin{eqnarray}
\nonumber\langle A_{h_j}^2u_j(h_j),u_j(h_j)\rangle \nonumber&=&\|A_{h_j}u_j(h_j)\|^2\\
\nonumber&\geq & |\langle A_{h_j}u_j(h_j),v_{m_j}(h_j)\rangle|^2\\
\nonumber&=& |\langle u_j(h_j),v_{m_j}(h_j)\rangle +\langle u_j(h_j),(A_{h_j}-\textrm{Id}) v_{m_j}(h_j)\rangle|^2\\
&\geq & \frac{1}{5\lambda^2 R^2}
\end{eqnarray}
for sufficiently large $j$, from \eqref{largeinnerproduct2} and the concentration of $v_{m_j}$ onto the torus $\Lambda_\omega$.

Now let $\nu$ be a semiclassical measure associated to a subsequence of the $u_{k_j}(h_j)$. We see that 
$$
\int \sigma(A) \, d\mu 
$$
is bounded below by $(5 \lambda^2 R^2)^{-1}$. 
By taking $A$ to have shrinking support in a neighbourhood of $\Lambda_\omega$, we see that $\nu$ has mass at least $(5 \lambda^2 R^2)^{-1}$ on $\Lambda_\omega$. 

Applying this argument with $\lambda \rightarrow \infty $ we establish the existence of such semiclassical measures for almost all $I_\omega \in E_\kappa$ and we are done.

One can apply Theorem \ref{mainthm} with $\kappa \rightarrow 0$ to obtain the following corollary.

\begin{corollary}
	\label{maincor}
	Under the same assumptions as Theorem \ref{mainthm}, for almost all nonresonant frequencies $\omega\in \cup_{\kappa>0} E_\kappa$, there exists a $t_0(\omega)>0$ such that for almost all $t\in (0,t_0)$ there exists a semiclassical measure associated to the eigenfunctions of $P_h(t)$ that has positive mass on $\Lambda_\omega$.
\end{corollary}

\section{Remarks on $\mathcal{C}^\infty$ case}

In the present article, the assumption was made that $P(x,\xi;t)\in G^{\rho,\rho,1}$ and that $M$ is a $G^\rho$ class manifold. 
This choice was made because to the authors knowledge, there does not appear to be any direct analogue of classical Birkhoff normal form (See \cite[Corollary~1.2]{popovkam}) in the literature for KAM systems that are $\mathcal{C}^\infty$ perturbations of $\mathcal{C}^\infty$ completely integrable systems. 
However, under the assumptions of the existence of such a Birkhoff normal form, a quantum Birkhoff normal form was obtained in the $\mathcal{C}^\infty$ setting by Colin de Verdi\`{e}re \cite{colin},  with the symbols $K^0,R^0$ of $\mathcal{C}^\infty$ regularity and quasimodes having $O(h^\infty)$ error terms. 
As we only require $O(h^{\gamma+1})$ quasimodes for the argument in this paper, the proof of Theorem \ref{mainthm} goes through in the $\mathcal{C}^\infty$ case in exactly the same manner.

\begin{theorem}
	\label{smooththm}
	Suppose $M$ is a compact boundaryless $\mathcal{C}^\infty$ surface, and suppose that $\mathcal{P}_h(t)$ is a family of self-adjoint elliptic semiclassical pseudodifferential operators of fixed positive differential order $m$, such that 
	\begin{itemize}
		\item The operator $\mathcal{P}_h(t)$ has full symbol real-valued, smooth in $t$, and in the standard  Kohn--Nirenberg symbol class;
		\item The principal symbol of $\mathcal{P}_h(t)$ is given by some  $P(x,\xi;t)\in \mathcal{C}^\infty(T^*M\times (-1,1))$;
		\item The Hamiltonian $P^0(x,\xi):=P(x,\xi;0)$ is, in some open set of phase space $T^* M$, non-degenerate and completely integrable; 
		\item Written in action-angle coordinates $(\theta, I) \in \T^n \times D$ for the Hamiltonian $P^0$, the vector fields 
		\begin{multline}
		\nabla_I H^0(I) \text{ and } \nabla_I \left(\int_{\T^2}\partial_t H(\theta,I;0)\, d\theta\right) \\ \text{ are linearly independent for $I\in D$ and all $h<h_0$,}
		\label{li-cond}\end{multline}
		where $H(\theta,I;t)$ denotes $P(x,\xi;t)$ written in the action-angle coordinates for $P^0$, and $H^0(I):=H(\theta,I;0)$.
	\end{itemize}
	
	Then there exists $t_0>0$ such that for almost all $t \in [0,t_0]$, and for almost all KAM tori $\Lambda_\omega=\T^n\times \{I_\omega\}$ with $\omega\in\Omega_\kappa$, there exists a semiclassical measure associated to the eigenfunctions of $\mathcal{P}_h(t)$ 
	that has positive mass (and hence scars) on $\Lambda_\omega$.
\end{theorem}





\section{Examples}

\subsection{The flat torus}
A fundamental example of a nondegenerate completely integrable system is the flat torus, $\T^2 = \R^2/2\pi \Z^2$ with the standard metric. If we denote the spatial coordinates by $(\theta_1, \theta_2)$ and their dual coordinates by $I_1, I_2$ then these form action-angle coordinates and the symbol of the Laplacian takes the form $I_1^2 + I_2^2$ which is obviously nondegenerate. Our result applies in particular to metric perturbations of the flat torus, provided that \eqref{li-cond} holds, which it clearly does locally for a generic perturbation. For example, if we take $H(\theta, I;t) = I_1^2+I_2^2+t\cos^2 \theta_1 I_1 I_2$ then this satisfies condition \eqref{li-cond} whenever $I_1 \neq I_2$. 

Another standard example of a completely integrable system is geodesic flow on an surface of revolution. In the case of the ellipsoid, this was checked to be nondegerate in \cite{ellipsoid}. Generic metric perturbations, or potential perturbations, of this metric can similarly be treated. 

\subsection{Spherical pendulum} Geodesic flow on the 2-sphere is certainly completely integrable, but it is degenerate, as every orbit is periodic. However, if we add to this system a potential which is the height function in the standard embedding into $\R^3$, then the system is still rotationally invariant, hence completely integrable, but now nondegenerate, as shown in \cite{pendulum}. Metric or potential perturbations of this system fall into the framework of this paper, provided that condition \eqref{li-cond} is satisfied locally. 

\subsection{Central potentials} Another standard completely integrable system is that of central potentials on $\R^2$, that is, the system on $T^* \R^2$ with Hamiltonian 
$$
h(x, y, \xi, \eta) = \xi^2 + \eta^2 + V(\sqrt{x^2 + y^2}).
$$
Again this is rotationally invariant and therefore completely integrable. The corresponding operator is the Schr\"odinger operator $h^2\Delta + V$. Although this is on a noncompact manifold, if $V(r) \to \infty$ as $r \to \infty$, then this operator has discrete spectrum and the phase space corresponding to any energy interval $[0, E]$ is compact, so our results apply directly. The system is nondegenerate for generic $V$. This follows from \cite[Chapter 2, Section 1.1]{arnoldcelestial},
where explicit formulae for the period $\tau$ between pericentre and apocentre, and for the aspidal angle $\Phi$, are given. Nondegeracy is equivalent to the condition that $(\tau, \Phi)$ are nondegenerate functions of the angular momentum $c$ and energy $h$, and it is straightforward to check from these explicit formulae that this is true after a generic perturbation of the potential. Our theorem applies for example to compactly supported (or decaying at infinity) potential perturbations satisfying \eqref{li-cond}. 

\appendix
\section{Gevrey class symbols and Birkhoff normal form}
\label{gevsymbsec}
In this appendix we collect the basic definitions of Gevrey functions and Gevrey symbol classes.

Let $D$ be a bounded domain in $\R^n$ and by let $X$ be either a bounded domain in $\R^n$ or the compact set $\T^n$.

\begin{defn}
\label{anisgevdef}
For $\rho,\rho'>1$, the anisotropic Gevrey class $G^{\rho,\rho'}(X\times D)$ is defined as the set of ${u\in \mathcal{C}^\infty(X\times D)}$ with
\begin{equation}
\sup_{\alpha,\beta}\sup_{(\theta,I)} |\partial_\theta^\alpha \partial_I^\beta u|L_1^{-\alpha}L_2^{-\beta}\alpha!^{-\rho}\beta!^{-\rho'}<\infty
\end{equation}
for some $L_1,L_2>0$.
\end{defn}
This definition can be extended for functions with more than two differing degrees of Gevrey regularity in the obvious manner.

We now fix the parameters $\sigma,\mu>1$ and $\varrho\geq \sigma+\mu-1$, and denote the triple $(\sigma,\mu,\varrho)$ by $\ell$.

\begin{defn}
	\label{formalsymboldefn}
	A formal Gevrey symbol on $X\times D$ is a formal sum
	\begin{equation}
	\label{formalsymbol}
	\sum_{j=0}^\infty p_j(\theta,I)h^j
	\end{equation}
	where the $p_j\in\mathcal{C}_0^\infty(X\times D)$ are all supported in a fixed compact set and there exists a $C>0$ such that
	\begin{equation}
	\sup_{X\times D} |\partial_\theta^\beta \partial_I^\alpha p_j(\theta,I)|\leq C^{j+|\alpha|+|\beta|+1}\beta!^\sigma\alpha!^\mu j!^\varrho.
	\end{equation}
\end{defn}

\begin{defn}
	A resummation of the formal symbol \eqref{formalsymbol} is a function $p(\theta,I;h)\in\mathcal{C}_0^\infty(X\times D)$ for $0<h\leq h_0$ with
	\begin{equation}
	\sup_{X\times D \times (0,h_0]} \left|\partial_\theta^\beta \partial_I^\alpha \left(p(\theta,I;h)-\sum_{j=0}^N  p_j(\theta,I)h^j\right)\right|\leq h^{N+1}C_1^{N+|\alpha|+|\beta|+2}\beta!^\sigma\alpha!^\mu (N+1)!^\varrho.
	\end{equation}
\end{defn}

\begin{lem}
	Given a formal symbol \eqref{formalsymbol}, one choice of resummation is
	\begin{equation}
	p(\theta,I;h):= \sum_{j\leq \epsilon h^{-1/\varrho}} p_j(\theta,I)h^j
	\end{equation}
	where $\epsilon$ depends only on $n$ and $C_1$.
\end{lem}

\begin{defn}
	We define the residual class of symbols $S_\ell^{-\infty}$ as the collection of realisations of the zero formal symbol.
\end{defn}

Writing $f\sim g$ if $f-g\in S_\ell^{-\infty}$, it then follows that any two resummations of the same formal symbol are $\sim$-equivalent. Gevrey symbols are precisely the equivalence classes of $\sim$.

\begin{defn}
	\label{selldef}
	We denote the set of equivalence classes by $S_\ell(X\times D)$.
\end{defn}

We can now introduce the pseudodifferential operators corresponding to these symbols.
\begin{defn}
	\label{gevpseudo}
	To each symbol $p\in S_\ell(X\times D)$, we associate a semiclassical pseudodifferential operator defined by
	\begin{equation}
	(2\pi h)^{-n}\int_{X\times \mathbb{R}^n}e^{i(x-y)\cdot \xi/h}p(x,\xi;h)u(y)\, d\xi\, dy.
	\end{equation}
	for $u\in \mathcal{C}_0^\infty(X)$.
\end{defn}

The above construction is defined modulo $\exp(-ch^{-1/\varrho})$, as for any $p\in S_\ell^{-\infty}(X\times D)$ we have
\begin{equation}
\|P_hu\|=O_{L^2}(\exp(-ch^{-1/\varrho}))
\end{equation}
for some constant $c>0$.
\begin{remark}
	The exponential decay of residual symbols is a key strengthening that comes from working in a Gevrey symbol class, as opposed to the standard Kohn--Nirenberg classes.
\end{remark}

An important feature of the Gevrey symbol calculus is that the symbol class $S_\ell(X\times D)$ is closed under composition. 

We conclude by noting that if $p\in S_{(\sigma,\sigma,2\sigma-1)}$, then $G^\sigma$ changes of variable preserve the symbol class of $p$. This coordinate invariance allows us to extend the Gevrey pseudodifferential calculus to compact Gevrey manifolds.

\bibliographystyle{plain}
\bibliography{kam.positive.mass}
\end{document}